\documentclass[12pt,a4paper]{amsart}
\usepackage{amsmath}
\usepackage{amssymb,amscd}
\usepackage{enumerate}

\usepackage{fullpage}
\usepackage{hyperref}

\theoremstyle{plain}
\newtheorem{theorem}{Theorem}[section]
\newtheorem{lemma}[theorem]{Lemma}
\newtheorem{proposition}[theorem]{Proposition}

\theoremstyle{definition}
\newtheorem{definition}[theorem]{Definition}
\newtheorem{example}[theorem]{Example}
\newtheorem{Conjecture}[theorem]{Conjecture}
\newtheorem{Remark}[theorem]{Remark}
\newtheorem{Question}[theorem]{Question}
\newtheorem{remark}[theorem]{Remark}

\def\Aut{\operatorname{Aut}}
\def\End{\operatorname{End}}
\def\Ker{\operatorname{Ker}}
\def\Der{\operatorname{Der}}
\def\LND{\operatorname{LND}}
\def\Frac{\operatorname{Frac}}
\def\Lie{\operatorname{Lie}}
\def\Spec{\operatorname{Spec}}
\def\alg{\operatorname{alg}}
\def\ord{\operatorname{ord}}
\def\ZZ{{\mathbb Z}}
\def\CC{{\mathbb C}}
\def\TT{{\mathbb T}}
\def\QQ{{\mathbb Q}}
\def\NN{{\mathbb N}}
\def\K{{\mathbb K}}
\def\G{{\mathbb G}}

\def\embed{\hookrightarrow}
\def\Aff{\operatorname{Aff}}
\def\reg{{\rm reg}}
\def\Im{\operatorname{Im}}

\def\LHC{\operatorname{LHC}}

\def\g{{\mathfrak g}}
\def\u{{\mathfrak u}}

\def\t{{\mathfrak t}}
\def\p{{\mathfrak p}}
\def\m{{\mathfrak m}}
\def\U{{\mathcal U}}

\def\SAut{\operatorname{\mathcal U}}

\makeatletter
\newcommand\thankssymb[1]{\lowercase{\textsuperscript{\@alph{#1}}}}
\makeatother
\newcommand{\name}[1]{\textsc{#1\/}}

\address{
  Kharkevich Institute for Information Transmission Problems\\
  19 Bolshoy Karetny per., 127994 Moscow, Russia
}

\address{
Moscow Institute of Physics and Technology (State University)\\
9 Institutskiy per., Dolgoprudny, Moscow Region, 141701, Russia 
}

\address{
National Research University Higher School of Economics\\
 20 Myasnitskaya ulitsa, Moscow 101000, Russia 
}

\email{a@perep.ru}

\address{\noindent Institut f\"{u}r Mathematik, Friedrich-Schiller-Universit\"{a}t Jena,   Jena 07737, Germany}
\email{andriyregeta@gmail.com}

\begin{document}

\title[When is the automorphism group nested?]{When is the automorphism group of an affine variety nested?}

\author{Alexander Perepechko\thankssymb{1}}
\thanks{\thankssymb{1}Supported by the Russian Foundation for Sciences (project no. 18-71-00153).}

\author{Andriy Regeta\thankssymb{2}}
\thanks{\thankssymb{2}Partially supported by SNF  (Schweizerischer Nationalfonds), project number P2BSP2\_165390.}

\maketitle

\begin{abstract} 
For an  affine algebraic variety $X$, we study the subgroup 
$\mathrm{Aut}_{\text{alg}}(X)$ of the group 
of regular automorphisms $\mathrm{Aut}(X)$
of $X$ generated by all the connected algebraic subgroups. We prove that 
$\mathrm{Aut}_{\text{alg}}(X)$ is nested, i.e., is a direct limit of algebraic subgroups of $\mathrm{Aut}(X)$, if and only if all the $\mathbb{G}_a$-actions on $X$ commute. Moreover, we describe the structure of such a group $\mathrm{Aut}_{\text{alg}}(X)$.

\end{abstract}

\section{Introduction}

It was proved in 1958 by \name{Matsusaka}  \cite{Mat58} that the neutral component $\Aut^\circ(Y)$ of the automorphism group 
of a projective irreducible algebraic  variety $Y$ is an algebraic group.
For affine algebraic varieties the situation is quite different. For example 
 the automorphism group $\Aut(\mathbb{A}^n)$ of an  affine $n$-space contains a copy of a polynomial ring in $n-1$ variables.
 Hence, there is no way to put a structure of an algebraic group on $\Aut(\mathbb{A}^n)$ for $n \ge 2$. In \cite{Sh66}  \name{Shafarevich} introduced the notion of \emph{ind-group}.
 It is known that for an affine  variety $X$ its automorphism group  $\Aut(X)$ has a natural structure of an ind-group (see  \cite[Section 5]{FK} and \cite[Section 2]{KPZ1} for details). 

The base field $\K$ is algebraically closed of zero characteristic, and the additive group of $\K$ is denoted by $\G_a$.  Through the whole paper
$X$ denotes an irreducible affine algebraic variety.
We call an element $g \in \Aut(X)$ \emph{algebraic} if there is an algebraic subgroup $G$ of the ind-group $\Aut(X)$.  We also denote by $\U(X)\subset \Aut(X)$ the (possibly trivial) subgroup generated by all the $\G_a$-actions.
It is also called the \emph{special automorphism group} and denoted by ${\mathrm SAut}(X)$.

 In  \cite{KPZ1} and \cite{PZ}   the neutral component $\Aut^{\circ}(X)$ of the ind-group of automorphisms $\Aut(X)$ of an affine surface $X$
 has been studied. Note that $\Aut^{\circ}(X)$ is a closed subgroup of $\Aut(X)$. 
 The equivalence of the following conditions is claimed:
 
 \begin{itemize}
\item  all elements  of $\Aut^{\circ}(X)$ are algebraic;
\item the subgroup $\Aut^{\circ}(X) \subset \Aut(X)$ is nested, i.e., is a direct limit of algebraic subgroups; 
\item $\Aut^\circ(X) = \TT \ltimes \U(X)$, where $\TT$ is a maximal subtorus of $\Aut(X)$. 
\end{itemize}
 
 Our intention is to prove this result independently in arbitrary dimension. 
Originally, we were motivated by  Conjecture~\ref{motiv-c} and Question~\ref{motiv-q}.
\begin{Conjecture}[\name{P.}--\name{Zaidenberg}, Feb.'13]\label{motiv-c}
An affine variety does not admit additive group actions if and only if the neutral component of the automorphism group is an algebraic torus.
\end{Conjecture}
The statement that the neutral component is a torus was proved in \cite[Theorem 1.3]{Kr} under the assumption that $\Aut^\circ(X)$ is finite-dimensional and in \cite[Proposition 3.2]{AG} for T-varieties satisfying certain conditions.
\begin{Question}[\name{Kraft}] \label{motiv-q}
{Which affine varieties have automorphism groups comprised of algebraic elements?}
\end{Question}
  
We provide a partial answer to Question~\ref{motiv-q} in Theorem~\ref{th:alg}.
In the direction of the intended generalization 
we prove in the present paper the following  statement.

\begin{theorem}\label{main}
Given an affine variety $X$, let $\Aut_{\alg}(X)$ be the subgroup of $\Aut(X)$ generated by all connected algebraic subgroups. The following conditions are equivalent:
\begin{enumerate}[(1)]
\item $\U(X)$ is abelian;
\item all elements  of $\Aut_{\alg}(X)$ are algebraic;
\item the subgroup $\Aut_{\alg}(X) \subset \Aut(X)$ is a closed nested ind-subgroup;
\item $\Aut_{\alg}(X) = \TT \ltimes \U(X)$, where $\TT$ is a maximal subtorus of $\Aut(X)$, and $\U(X)$ is closed in $\Aut(X)$. 
\end{enumerate}
 \end{theorem}

\begin{Remark}
 If $\dim X \ge 2$, then $\U(X)$ is either trivial or infinite-dimensional. 
 Indeed, if there exists a $\G_a$-action corresponding to a locally nilpotent derivation $\partial$, 
 then we have an infinite-dimensional unipotent subgroup $\exp((\ker\partial)\cdot\partial)\subset\U(X)$, e.g. see \cite[Theorem 6.3]{KPZ}.
\end{Remark}

We expect that  Theorem~\ref{main}  holds if we replace  $\Aut_{\alg}(X)$   by $\Aut^\circ(X)$. In particular, we formulate the following extension of Conjecture~\ref{motiv-c}.

\begin{Conjecture}
If $X$ is an affine variety, then $\U(X)$ is abelian if and only if $\Aut^\circ(X)$ is nested.
\end{Conjecture}

\emph{Acknowledgements}: We would like to thank \name{Hanspeter Kraft} and \name{Mikhail Zaidenberg}  for motivation, useful discussions, and numerous remarks. 
We also thank the referees for the useful remarks.

\section{Preliminaries}

\subsection{ Derivations and group actions}
Recall that $X$ is an irreducible affine algebraic variety.
A derivation $\delta$ is called \emph{locally finite} if it acts locally finitely on ${\mathcal O}(X)$, i.e., for any $f \in {\mathcal O}(X)$ there is a  finite-dimensional vector subspace
$V \subset {\mathcal O}(X)$ such that $f \in V$ and  $V$ is stable under action of $\delta$.
A derivation $\delta \in \Der({\mathcal O}(X))$ is called \emph{semisimple} if there exists a basis 
$\{ f_i \mid i \in \mathbb{N} \}$ of the vector space ${\mathcal O}(X)$ such that $\delta(f_i) \in \K f_i$. Finally, 
a derivation $\delta \in \Der({\mathcal O}(X))$ is called \emph{locally nilpotent} if for any  $f \in {\mathcal O}(X)$ there exists $n \in \mathbb{N}$ (which depends on $f$)
such that $\delta^n(f) = 0$. 
There is a one-to-one correspondence between locally nilpotent derivations on ${\mathcal O}(X)$ and $\G_a$-actions on 
$X$ given by the map $\delta\mapsto\{t\mapsto\exp(t\delta)\}$.
We denote the set of locally nilpotent derivations (LNDs) on ${\mathcal O}(X)$ by $\LND(X)$. 
We call two LNDs $\partial_1,\partial_2$ \emph{equivalent} if their kernels coincide. By \cite[Principle 12]{Fr}, equivalence of $\partial_1$ and $\partial_2$ implies that $\partial_1=c\partial_2$ for some $c\in\Frac\ker\partial_1$. If $\partial_1$ and $\partial_2$ are equivalent,  we call the corresponding $\G_a$-actions $\exp(t\partial_1)$ and $\exp(t\partial_2)$ equivalent as well. Note that these $\G_a$-actions have the same general orbits and hence commute.  Indeed, a general orbit is an affine line, and any two unipotent elements of $\Aut({\mathbb A}^1)\cong\G_a\ltimes\G_m$ commute.

An element $u\in\Aut(X)$ is called \emph{unipotent} if $u=\exp(\partial)$ for some $\partial\in\LND(X)$. An (ind-)subgroup $U\subset\Aut(X)$ is called \emph{unipotent} if each $u\in U$ is unipotent. 
\begin{definition}
 We denote the Lie subalgebra of the Lie algebra of derivations $\Der({\mathcal O}(X))$ on $X$ generated by all LNDs by $$\u(X)=\langle \partial\mid\partial\in\LND(X)\rangle,$$ and the automorphism subgroup generated  by the   unipotent elements by 
 \[\U(X)=\langle \exp(\partial)\mid \partial\in\LND(X)\rangle\subset\Aut(X).\]
\end{definition}

\begin{lemma}\label{lem:equivalent}
The unipotent-generated subgroup $\U(X)$ is abelian if and only if all LNDs on $X$ are equivalent.
\end{lemma}
\begin{proof}
 If all LNDs are equivalent, then they share the same kernel and coincide up to a multiplication by kernel elements. Thus, their exponents commute and comprise $\U(X)$.
 
 Assume the contrary, let $\partial_1$ and $\partial_2$ be two non-equivalent LNDs. If $\partial_1$ and $\partial_2$ do not commute, then the corresponding $\G_a$-actions do not commute too and the proof follows. Hence, we can assume that $\partial_1$ and $\partial_2$ commute. Since $\partial_1$ and $\partial_2$ are not equivalent, 
 there exists 
$f \in \ker \partial_1$ that does not belong to $\ker \partial_2$. Hence, $[\partial_2,f \partial_1] =
\partial_2(f)\partial_1 \neq 0$. Non-commutativity of the
LNDs $\partial_2$ and $f \partial_1$ implies non-commutativity of the $\G_a$-actions $\{ \exp(t\partial_2)\mid t\in\K \}$ and 
 $\{ \exp(t f\partial_1)\mid t\in\K \}$. The proof follows.
\end{proof}

\subsection{Ind-groups}

The notion of an ind-group goes back to \name{Shafarevich} who called these objects infinite dimensional groups (see \cite{Sh66}). 
We refer to \cite{FK} and  \cite[Section 2]{KPZ1}  for basic notions in this context.

\begin{definition}
By an affine \emph{ind-variety}  we mean an injective limit $V=\varinjlim V_i$ of an ascending sequence $V_0 \embed V_1 \embed V_2 \embed \ldots$ 
such that the following holds:
\begin{enumerate}[(1)]
\item $V = \bigcup_{k \in \NN} V_k$;
\item each $V_k$ is an affine algebraic variety;
\item for all $k \in \NN$ the embedding $V_k \embed V_{k+1}$ is closed in the Zariski topology.
\end{enumerate}
\end{definition}

For simplicity we will  call an affine ind-variety  simply an ind-variety.

An ind-variety $V$ has a natural \emph{topology}: a subset $S \subset V$ is called open, resp. closed, if $S_k := S \cap  V_k \subset V_k$ is open, resp. closed, for all $k \in \mathbb{N}$.  A  closed subset $S \subset V$ has a natural structure of an ind-variety and is called an ind-subvariety.

The product of  ind-varieties is defined in the obvious way. 
A \emph{morphism} between ind-varieties $V = \bigcup_k V_k$ and $W = \bigcup_m W_m$ is a map $\phi: V \to W$ such that for every $k \in \mathbb{N}$ there is an $m \in \mathbb{N}$ such that 
$\phi(V_k) \subset W_m$ and that the induced map $V_k \to W_m$ is a morphism of algebraic varieties. 
 This allows us to give the following definition.

\begin{definition}
An ind-variety $G$ is said to be an \emph{ind-group} if the underlying set $G$ is a group such that the map $G \times G \to G$,  $(g,h) \mapsto gh^{-1}$, is a morphism.
\end{definition}

A \emph{closed subgroup}  $H$ of $G$ is a subgroup that is also a closed subset. Then $H$ is again an ind-group with respect to the induced ind-variety structure.
 A closed subgroup $H$ of an ind-group $G = \varinjlim G_i$ is called an \emph{algebraic subgroup} if $H$ is contained in $G_i$ for some $i$.  From \cite[Proposition 5.6.5(1)]{FK} it follows  that algebraic subgroups of $\Aut(X)$ are exactly algebraic groups that act regularly on $X$.

The next result can be found in   \cite[Section 5]{FK} and \cite[Section 2]{KPZ1}.
\begin{proposition}\label{ind-group}
Let $X$ be an affine variety. Then $\Aut(X)$ has the structure of an ind-group such that  a regular action of an algebraic group $G$ on $X$ induces an ind-group homomorphism $G \to \Aut(X)$.
\end{proposition}

\begin{definition}
An element $g \in \Aut(X)$ is called \emph{algebraic} if there is an algebraic subgroup $G \subset \Aut(X)$ such that $g \in G$. 
\end{definition}

\begin{definition}
 An ind-group $G$ is called \emph{nested} if $G=\varinjlim G_i$, where $G_i$ is an algebraic group and $G_i\subset G_{i+1}$ is a closed subgroup for $i=1,2,\ldots$.
\end{definition}
If $G\subset\Aut(X)$ is nested, then there exists a Levi decomposition of the neutral component $G^\circ=L\ltimes U$, where $L$ is a reductive algebraic group and $U$ is a normal ind-subgroup which consists of unipotent elements, see \cite[Theorem 2.11]{KPZ1}. In addition, there exists an algebraic subgroup $H\subset G$ with same orbits on $X$, i.e., $G\cdot x=H\cdot x$ for any $x\in X$, see  \cite[Proposition 2.17]{KPZ1}.

\subsection{Lie algebras of ind-groups}\label{Lie}
For any ind-variety $V = \bigcup_{k \in \mathbb{N}} V_k$ we can define the tangent space in $x \in V$ in the obvious way: we  have  $x \in V_k$ for $k \ge k_0$, and
$T_x V_k \subset T_x V_{k+1}$ for $k \ge k_0$, and then define
$$
  T_x V := \bigcup_{k \ge k_0} T_x V_k,
$$
which is a vector space of countable dimension.

For an  ind-group $G$, the tangent space $T_e G$ has a natural structure of a Lie algebra which is  denoted by $\Lie G$
(see \cite[Section 4]{Ku} and \cite[Section 2]{FK} for details). By $\overline{\Aut_{\text{alg}}(X)} \subset \Aut(X)$ we denote 
the closure of the subgroup $\Aut_{\text{alg}}(X)$
in $\Aut(X)$ generated by all connected algebraic subgroups.
By \cite[Theorem 0.3.2]{FK}  
there is an injective antihomomorphism from the Lie algebra
$\Lie \Aut(X)$ 
into the Lie algebra $\Der({\mathcal O}(X))$ of derivations on $X$.
  From now on,	we  will  always	identify	$\Lie \Aut(X)$ and $\Lie \overline{\Aut_{\text{alg}}(X)}$	with	their	images  in  $\Der({\mathcal O}(X))$.   Note  that $\Lie \overline{\Aut_{\text{alg}}(X)}$ contains all locally finite derivations  because  each such derivation $\delta$ is contained in
$\Lie G$ for some connected algebraic subgroup $G \subset \Aut(X)$.

\section{The case: $\U(X)$ is not abelian}
Provided that the unipotent-generated subgroup  $\U(X)$ is not abelian, by Lemma~\ref{lem:equivalent} there exist non-equivalent $\G_a$-actions on $X$.
 The aim of this section is to prove the following result.
\begin{proposition}\label{nonlocfin}
Assume that an affine variety $X$  admits two non-equivalent $\mathbb{G}_a$-actions.
Then

(1) there exists a derivation  $\partial$ in the linear span of $\LND(X)$  which is not locally finite.

(2) there exists a non-algebraic element in $\U(X)$.
\end{proposition}

 \begin{remark}
A  variety $X$ as in this Proposition \ref{nonlocfin} cannot be of dimension $\le1$, otherwise all LNDs are equivalent. Thus, $\dim X\ge2$.
\end{remark}
Let $\partial_1$, $\partial_2$ be two  locally  nilpotent  derivations  corresponding to two non-equivalent $\mathbb{G}_a$-actions $H_1, H_2$ on $X$, respectively,
$\p_i=(\Ker\partial_i)\cap(\Im\partial_i)$, $i=1,2$ their plinth ideals, and $v_1,v_2$ their corresponding vector fields.

Let us consider a fibration $X\to {\mathbb A}^1,\, p\mapsto f(p)$ for some  $f\in\ker\partial_1$ such that $f\notin\ker\partial_2$.
Then $H_1$-orbits lie in its fibers, but general $H_2$-orbits do not. 
Hence  $v_1(p)$ and $v_2(p)$ are linearly independent at a general point.
Since $V(\p_1),V(\p_2)$ are proper closed subsets of $X$, we can take a smooth point $p\in X_\reg\setminus(V(\p_1)\cup V(\p_2))$ such that $v_1(p)$ and $v_2(p)$  are linearly independent.

Consider local coordinates $(x_1,\ldots,x_n)$, where $n=\dim X$, at $p$ such that 
\[v_1(p)=(1,0,0,\ldots,0),\quad v_2(p)=(0,1,0,\ldots,0).\]

Let $\m_p$ be the maximal ideal of 
${\mathcal O}(X)$ that corresponds to $p \in X$. We operate in the $\m_p$-adic completion of the local ring at $p$
$$\hat {\mathcal O}_p(X) =  \varprojlim_k {\mathcal O}_p(X)/\m_p^k{\mathcal O}_p(X).$$

 We may assume in this section that 
${\mathcal O}(X)\subset{\mathcal O}_p(X)\subset\hat{\mathcal O}_p(X)$ because ${\mathcal O}_p(X)$ is a localization of ${\mathcal O}(X)$ and there is a canonical embedding ${\mathcal O}_p(X)\subset\hat{\mathcal O}_p(X)$. Moreover, each derivation of ${\mathcal O}(X)$ is uniquely extended to ${\mathcal O}_p(X)$ and each derivation of ${\mathcal O}_p(X)$ is uniquely extended to a derivation of $\hat{\mathcal O}_p(X)$ (see e.g., \cite[Tag 07PE]{Stacks}), so for each $\delta\in\Der {\mathcal O}(X)$ we denote its extension by $\hat\delta\in\Der\hat{\mathcal O}_p(X)$.

Since $p$ is smooth, $\hat {\mathcal O}_p(X) = k[\![x_1,x_2,\ldots,x_n]\!]$
is a formal power series ring (by the Cohen structure theorem, e.g., see  \cite{Co}).
Thus, we have a natural $\ZZ_{\ge0}$-grading on $\hat {\mathcal O}_p(X)$ by the minimum degree, which in turn induces the $\ZZ_{\ge-1}$-grading on $\Der \hat {\mathcal O}_p(X)$ via the formula $\deg \partial =\deg \partial h - \deg h $ for a homogeneous derivation $\partial$ and any homogeneous element $h\in\hat {\mathcal O}_p(X)$. 
Let $f$ be an element of either $\hat {\mathcal O}_p(X)$ or $\Der \hat {\mathcal O}_p(X)$. We denote by $\LHC(f)$ the  homogeneous component of lowest degree and by $f_{(d)}$ the $d$th homogeneous component.
By our convention $\hat\partial_i\in\Der \hat {\mathcal O}_p(X)$ is the derivation induced by $\partial_i$, $i=1,2$.  Since 
$v_1(p) = (1,0,0,...,0)$ and 
$v_2(p) = (0,1,0,...,0)$ indicate the lowest (linear) homogeneous components of $\hat\partial_1,\hat\partial_2$ respectively,  
we have $\LHC(\hat\partial_i)= \frac{\partial}{\partial x_i}$, $i=1,2$.
\begin{lemma}\label{lem:ker}
\begin{enumerate}
\item $\LHC(g)\in\K[\![x_2,\ldots,x_n]\!]$ for any $g\in\ker\hat\partial_1$.
\item The  map $\ker\hat\partial_1\to\K[\![x_2,\ldots,x_n]\!]$ which maps 
$g(x_1,...,x_n) \in \ker\hat\partial_1$  to   $g(0,x_2,...,x_n)$ is an isomorphism of algebras.
\end{enumerate}
The same holds if we switch $x_1$ with $x_2$ and $\hat\partial_1$ with $\hat\partial_2$ respectively.
\end{lemma}
\begin{proof}
The first assertion is straightforward:
\[\hat\partial_1g=0 \implies \frac{\partial\LHC(g)}{\partial x_1}=0 \implies \LHC(g)\in \K[\![x_2,\ldots,x_n]\!].\]

The second assertion is that for any $g_0\in\K[\![x_2,\ldots,x_n]\!]$ there exists a unique 
element $g\in\Ker\hat\partial_1$ such that $g_0=g(0,x_2,\ldots,x_n)$. 
Let us split the equation $\hat\partial_1 g=0$ into homogeneous parts:
\[0=(\hat\partial_1 g)_{(k)}=(\hat\partial_1   g_{(0)}+\ldots+\hat\partial_1  g_{(k)})_{(k)}+\frac{\partial}{\partial x_1} g_{(k+1)}, \qquad k=0,1,\ldots\]
Thus, $\frac{\partial}{\partial x_1} g_{(k+1)}=-\sum_{i=0}^k 
(\hat\partial_1 g_{(i)})_{(k)}$, and $g_{(k+1)}$ is uniquely determined by lower homogeneous components up to $x_1$-free monomials. 
But the $x_1$-free monomials of $g$ comprise exactly $g(0,x_2,\ldots,x_n)$.
Thus, all homogeneous components of $g$ are uniquely constructed by induction on the degree from the $x_1$-free part $g(0,x_2,\ldots,x_n)=g_0$.

The statement for $\hat\partial_2$ is analogous.
\end{proof}

 \begin{lemma}\label{pr:x1+x2}
For any $d>1$ there are elements $f_i\in\ker\partial_i$, $i=1,2$ such that $\partial=f_1\partial_1+f_2\partial_2\in\u(X)$ satisfies
\[\LHC(\hat \partial)=x_2^d  \frac{\partial}{\partial x_1} + x_1^d \frac{\partial}{\partial x_2}.\]
\end{lemma}
\begin{proof}
By Lemma~\ref{lem:ker}, we may take $g_1\in\ker\hat\partial_1$ such that $\LHC(g_1)=x_2^d$.
Since $\p_1\nsubseteq\m_p$, by \cite[Lem.~3.2]{Mi},
$\ker\hat\partial_1$ equals the $(\m_p\cap\ker\partial_1)$-adic completion of $\ker\partial_1.$
Thus, the images of $\ker\partial_1$ and $\ker\hat\partial_1$ in $\hat{\mathcal O}_p(X)/\hat\m_p^{d+1}={\mathcal O}_p(X)/\m_p^{d+1}{\mathcal O}_p(X)$  coincide, where $\hat\m_p=\m_p\hat{\mathcal O}_p(X)$.
Therefore, there exists $f_1\in\ker\partial_1$ such that $\LHC(f_1)=\LHC(g_1)=x_2^d$.
Analogously, there exists $f_2\in\ker\partial_2$ such that $\LHC(f_2)=x_1^d$.
The statement follows.
\end{proof}

\begin{proof}[Proof of Proposition \ref{nonlocfin}]
(1) Let us take a derivation $\partial$ as in Lemma~\ref{pr:x1+x2} for $d=2$, i.e., $\LHC(\hat\partial)=x_2^2  \frac{\partial}{\partial x_1} + x_1^2 \frac{\partial}{\partial x_2}$. It is enough to prove that
$\partial$ is not locally finite.
Let $f\in{\mathcal O}(X)$ be such that $\LHC(f)=x_1+x_2$.
Then for each $k\ge 1$
\[\LHC(\hat\partial^{k-1} f)=\sum_{i=0}^k c_{k,i} x_1^ix_2^{k-i},\]
where 
$c_{k,i}\in\ZZ_{\ge 0}$ 
and $\sum_{i = 0}^k c_{k,i} > 0$. Thus, $\ord \partial^{k-1} f=k$, hence a sequence $\{\partial^k f\mid k=0,1,\ldots\}$  spans an infinite-dimensional subspace of ${\mathcal O}(X)$.

(2) 
In terms of Lemma~\ref{pr:x1+x2}, let \[g = \exp(f_1\partial_1)\circ\exp(f_2\partial_2).\]
Then $g$ belongs to $\U(X)$,  fixes $p$ and induces an automorphism $g^*$
 of $\hat{\mathcal O}_p(X)$ that preserves the subalgebra ${\mathcal O}(X)$. 
 A direct calculation shows that the linear operator $h=g^*-\mathrm{id}\in\End\hat{\mathcal O}_p(X)$  satisfies the following equality: \[\LHC(h(x_1^{a_1}x_2^{a_2}))=  a_1 x_1^{a_1 - 1} x_2^{d + a_2} + a_2 x_1^{d + a_1}x_2^{a_2-1},\]
where  $x_i^{-1}$ is zero by definition, $i = 1,2$.
 Moreover, $h(x_i)$ for $i>2$ is of degree at least $d+1$, if nonzero.
Hence, for a given  $f\in \hat{\mathcal O}_p(X)$ such that $\LHC(f)=P(x_1,x_2)$ is a polynomial of degree $s>0$ with positive integer coefficients,  $\LHC(h(f))$ is again a polynomial in $x_1,x_2$ of degree $s+d-1$ with positive integer coefficients.

  Let us take $f\in{\mathcal O}(X)$ such that $\LHC(f)=x_1$
and let $F \subset{\mathcal O}(X)$ be a minimal subspace that contains $f$ and is $h$-stable. Since $h^i(f) \in F$  and 
$\deg(\LHC(h^i(f))) = 1+i(d-1)$ for any $i\in\ZZ_{\ge0}$,
$F$ is infinite-dimensional. We claim that $g$ is not algebraic. Indeed,  if $g$ were algebraic, then $g^*$ would act locally finitely on ${\mathcal O}(X)$, and so would $h$.  The claim follows.
\end{proof}

\begin{example}
 By Jung--Van der Kulk's theorem, the automorphism group of the affine plane $X={\mathbb A}^2$ equals the amalgamated product
\[\Aut({\mathbb A}^2)=\Aff({\mathbb A}^2)\star_{C}\Aut_{\pi_1}({\mathbb A}^2),\]
where $\Aff({\mathbb A}^2)$ is the subgroup of affine transformations,
\[
\Aut_{\pi_1}({\mathbb A}^2) = \{(x,y)\mapsto (ax+P(y),by+c)\mid, a,b\in\K^\times,c\in\K, P\in\K[y]\}
\]
is the subgroup preserving the projection $\pi_1\colon{\mathbb A}^2\to{\mathbb A}^1,\;(x,y)\mapsto x,$ and $C$ is their intersection. Thus, if $u\in \Aut_{\pi_1}({\mathbb A}^2)\setminus C$ and  $g\in\Aff({\mathbb A}^2)\setminus C$, then $u\cdot gug^{-1}$ is a product of two unipotent elements which is not algebraic.
\end{example}
\begin{remark}
There exists an affine surface $X$  (see \cite{BD}) with the \emph{huge} automorphism group, i.e., such that $\Aut(X)/\Aut_{\alg}(X)$ is not countably generated. 
We believe that $\Aut^\circ(X)/\Aut_{\alg}(X)$ is uncountably generated as well.
\end{remark}

\section{The case: 
$\U(X)$
is abelian }
We denote $\g=\Lie\Aut(X) \subset \Der({\mathcal O}(X))$.
The following lemma is  well known  and appeared in similar form in \cite[Lemma 3.1]{FZ-uniqueness} and \cite[Theorem 2.1]{AG}.
\begin{lemma}\label{l:hom-lnd}
Assume that $\g$ is $\ZZ^r$-graded for $r>0$ and
consider a locally finite element $z\in\g$ that does not belong to the zero component $\g_0$. Then 
there exists a locally nilpotent homogeneous component of $z$ of non-zero weight.
\end{lemma}
\begin{proof}
Let us take the convex hull $P(z)\subset \ZZ^r\otimes\QQ$ of component weights of $z$.
Then for any non-zero vertex $v\in P(z)$ the corresponding homogeneous component is locally nilpotent.
The details are left to the reader.
\end{proof}

In this section we assume that $\U(X)$ is abelian. 
The next lemma is an adaptation of \cite[Lemma~3.6]{FZ-uniqueness} for locally finite elements.
\begin{lemma}\label{lem:semisimple}
Let $\delta$ be a locally finite derivation, and $\partial$ be a locally nilpotent derivation.
If $\U(X)$ is abelian, then $\delta-\partial$ is locally finite. \end{lemma}
\begin{proof}
Since $\delta \in \Der({\mathcal O}(X))$ is a locally finite element, there is the Jordan decomposition
into a sum of a locally nilpotent element $\delta_n$ and a semisimple element $\delta_s$ that belongs to the Lie algebra of some torus $T$,  e.g., see \cite[Section 2]{FZ-uniqueness} or \cite[Prop. 7.6.1]{FK}. 
The character lattice $M\cong \ZZ^r$ of $T$ induces an $M$-grading ${\mathcal O}(X)=\bigoplus_{\chi\in M} {\mathcal O}(X)_\chi$. 
The map $\chi\colon T\to\K^\times$ induces the tangent map $\Lie T\to\K$, which we denote by the same letter.
So, $\delta_s a=\chi(\delta_s) a$  for $a\in {\mathcal O}(X)_\chi$. Consider the homogeneous decomposition of $\partial$ 
with respect to this grading, i.e., 
$\partial = \sum_{\chi \in M} \partial_\chi$, where $[\delta_s,\partial_\chi] = \chi(\delta_s) \partial_\chi$;  $\chi$ is called the degree of $\partial_\chi$. 
Note that 
$[\delta,\partial']=[\delta_s,\partial']$ for any LND $\partial'$, since $[\delta_n,\partial']=0$.

If $\partial = \partial_0$, then 
$[\delta,\partial] = 0$ and the difference of two commuting locally finite derivations $\delta - \partial$ is again locally finite.
If $\partial \neq \partial_0$, then 
by Lemma~\ref{l:hom-lnd}, there exists a locally nilpotent homogeneous component $\partial_v$ of $\partial$, $v \neq0$. 
By Lemma~\ref{lem:equivalent}, 
 for each $\chi\in M$ we have
$\partial_\chi=c_\chi\partial_v$ for some $c_\chi$ from the field of fractions of $\ker\partial$; thus, $c_\chi$ is a homogeneous rational function of  degree $\chi-v$.

So, 
\[ \quad
[\delta,\partial]=[\delta_s,\sum_{\chi \in M}\partial_\chi]=\sum_{\chi \in M}\chi(\delta_s)\partial_\chi =   
\left(\sum_{\chi \in M}\chi(\delta_s)c_\chi\right)\partial_v.
\]
Taking $\partial'=\sum_{\chi\neq 0}\frac{c_\chi}{\chi(\delta_s)}\partial_v$,
we have  $[\delta,\partial']= \sum_{\chi\neq0} c_\chi\partial_v=\partial-\partial_0$, where the zero component $\partial_0= c_0 \partial_{v}$ might be trivial.

Derivations $[\delta,\partial']$ and $\partial'$ are locally nilpotent, hence commute. 
Thus, applying \cite[Lemma~2.4]{FZ-uniqueness} to $\delta$ and $-\partial'$, we conclude that
$\delta - \partial+\partial_0=\exp(\partial')  \delta \exp(-\partial')$ is locally finite.
Since $\partial_0$ commutes with both
$\delta$ and $\partial-\partial_0$,
 the difference of locally finite elements $\partial_0$ and $\delta-\partial+\partial_0$ is again locally finite.
The claim follows.
\end{proof}

Recall that $\u=\langle\partial\mid\partial\in\LND(X)\rangle$
is the Lie subalgebra of $\Der({\mathcal O}(X))$ generated by LNDs. By  $\t$ we denote  the Lie algebra of a maximal subtorus $\TT\subset \Aut(X)$.

\begin{proposition}\label{th:lie}
If $\U(X)$ is abelian, then every locally finite derivation on $X$ belongs to the semidirect product of $\t$ and $\u$.
\end{proposition}
\begin{proof}
First note that any locally finite derivation 
on $X$ belongs to $\g = \Lie \Aut(X)$.
Now, the adjoint action of $\t$ on $\g$ induces a grading on $\g$ by the character lattice $M\cong\ZZ^r$, which we fix. We
proceed by induction on the number of homogeneous components of $z$. 
If $z\in \g_0$, then $z$ commutes with $\t$. Thus, the semisimple part $z_s$ commutes with $\t$ and  due to the maximality of $\TT$, $z_s$ belongs to $\t$. Therefore, $z=z_s+z_n$ belongs to the semidirect product of $\t$ and $\u$.
If $z\notin\g_0$, then there exists a locally nilpotent homogeneous component $z_v$ of $z$ (see Lemma~\ref{l:hom-lnd}).
Hence, $z-z_v$ is locally finite by Lemma~\ref{lem:semisimple}, which belongs to
the semidirect product of $\t$ and $\u$  by the induction hypothesis. Therefore, $z=(z-z_v)+z_v$ also belongs to the semidirect product of $\t$ and $\u$.
\end{proof}

\begin{proposition}\label{Prop4.4}
If $\U(X)$ is abelian, then   the group $\Aut_{\alg}(X)$ coincides with $\TT\ltimes\U(X)$ and $\Aut_{\alg}(X)$ is a closed normal subgroup of $\Aut^\circ (X)$.
\end{proposition}
\begin{proof}
 Let $G \subset \Aut^\circ(X)$ be a connected algebraic
 subgroup. Then the Lie algebra 
 $\Lie G$ consists of locally finite derivations and, by Proposition \ref{th:lie}, $\Lie G \subset \t \oplus \u$ as a vector space.

 Since $\u$ is $\t$-stable, there exists a decomposition $\u=\bigoplus_{i=1}^\infty \K\partial_i$ such that $[\t,\partial_i]\subset \K\partial_i$ for each $i=1,2,\ldots$. Thus, $\Lie G\subset \t\oplus \bigoplus_{i=1}^k\K\partial_i$ for some $k$. Therefore, $G\subset \TT \ltimes U_k$, where $U_k=\exp(\bigoplus_{i=1}^k\K\partial_i)$ is a finite-dimensional $\TT$-stable unipotent group, see also \cite[Remark 17.3.3]{FK}.
 This means that $\Aut_{\alg}(X)$ coincides with
$\TT \ltimes \U(X)$.  
 
  The group $\Aut_{\alg}(X)$ is normal in $\Aut^\circ(X)$, because the set of connected algebraic subgroups of $\Aut(X)$ is stable under conjugation.

Let us prove that $\Aut_{\alg}(X)$ is a closed ind-subgroup in $\Aut(X)$. 
Assume that $\U(X)$ is non-trivial, otherwise $\Aut_{\alg}(X)=\TT$ and the statement follows.  

Fix some $\partial_0\in\LND(X)$ and let $x\in{\mathcal O}(X)$ be such that $\partial_0(x)\neq0$.
Fix also a nonzero element of the plinth ideal $b\in\Im(\partial_0)\cap\ker(\partial_0)$. 
Then any LND $\partial\in\LND(X)$ is equal to $\frac{a}{b}\partial_0$ for some $a\in\ker(\partial_0)$. 
Indeed, there exists $f$ such that $\partial_0(f)=b$, hence $\partial = \frac{\partial(f)}{b}\partial_0$, see also \cite[Principle 12]{Fr}.

Let $H\subset\Aut(X)$ be some closed algebraic subset, then $V=\langle H\circ x\rangle_\K$ and $V_\TT = \TT\circ V$ are finite-dimensional subspaces in ${\mathcal O}(X)$.
Take $g\in H\cap\Aut_{\alg}(X)$, then $g\circ x\in V$. 
Consider the decomposition $g=t\cdot u$, where $t\in\TT, u\in\U(X)$. Then $u\circ x\in V_\TT$.

There exists $\partial\in\LND(X)=\u$ such that $u=\exp(\partial)$. 
Let $\partial = \frac{a}{b}\partial_0$, then
\[\exp(\partial_0)\circ x=\sum_{i=0}^s
\frac{\partial_0^i(x)}{i!}\quad\mbox{and}\quad u\circ x = \sum_{i=0}^s
\frac{a^i\partial_0^i(x)}{b^ii!},\]
where $s$ is such that $\partial_0^s(x) \neq 0$ and $\partial_0^{s+1}(x)=0$.

Choose an embedding $X\embed{\mathbb A}^n$ for some $n$ to define the degree on ${\mathcal O}(X)$ as usual: $\deg(f)=\min\{\deg(F)\mid F\in{\mathcal O}({\mathbb A}^n), F|_X = f\}$ for $f\in{\mathcal O}(X)$. 
  Then $\deg (u\circ x)\le \max \{ \deg(v)\mid v \in  V_\TT \}$. 
  Therefore, if the degree of $u\circ x$ is not less than the degree of any summand of its decomposition $\frac{a^i\partial_0^i(x)}{b^ii!}$, 
  then for any $i$

  \[ \deg(u\circ x)= i(\deg(a)-\deg(b))+\deg(\partial_0^i(x))\le \max_{v\in V_\TT} (\deg(v)).\]
Otherwise, there is a cancellation in the decomposition of $u\circ x$, hence two summands have the same degree, and for some different $i,j$
  \[
  i(\deg(a)-\deg(b))+\deg(\partial_0^i(x)) =
  j(\deg(a)-\deg(b))+\deg(\partial_0^j(x)).
  \]
  In both cases there holds
  \[
  \deg(a)-\deg(b)\le \max_{v\in V_\TT} (\deg(v)) + \max_{i\le s} \deg(\partial_0^i(x)),
  \] 
  so $\deg(a)$ is bounded by some number, say, $N$ for any $u = \exp(\frac{a}{b}\partial_0) \in \U(X)$ such that $u \circ x\in  V_\TT$. 
   
Then 
\[U_N = \{\exp\left(\frac{a}{b}\partial_0\right)\mid a\in\ker\partial_0, \deg(a)\le N, \frac{a}{b}\partial_0\in\LND(X)\}\]
is a finite-dimensional subgroup in $\U(X)$ such that $H\cap \Aut_{\alg}(X) = H\cap (\TT\ltimes U_N)$, which is closed.  
Since $H$ is an arbitrary closed algebraic subset, $\Aut_{\alg}(X)$ is a closed ind-subgroup.
\end{proof}

\begin{example}
Given an algebraic curve $C$, the configuration space $X=\mathcal{C}^n(C)$ is the algebraic
variety consisting of all $n$-point subsets of $C$. By \cite[Theorem 1.2]{LZ}, for $n>2$ the neutral component $\Aut^\circ(X)$ is nested and equals the semidirect product $\TT\ltimes\U(X)$, where $\TT$ is a two-dimensional algebraic torus and $\U(X)$ is abelian.
\end{example}

\section{Conclusion}
In the following theorem we reformulate our result  geometrically in terms of fibrations.
We define an $\mathbb{A}^1$-fibration on
X to be a dominant morphism $f \colon X \to Y$ whose general fibers  are isomorphic to the affine line $\mathbb{A}^1$ (see for example \cite{GMM12}).
A $\G_a$-action $H$ on an affine variety $X$ induces the quasi-affine variety $Y=\Spec {\mathcal O}(X)^H$, e.g. see \cite[Theorem 1]{Win03}, and the $\mathbb{A}^1$-fibration $\mu \colon X \to Y$, on the fibers of which $H$ acts.
Moreover,
 equivalent $\G_a$-actions induce the same fibration.

\begin{theorem}\label{th:alg}
If  $\Aut_{\alg}(X)$ consists of algebraic elements, then one of the following holds:
\begin{enumerate}[(i)]
\item there exists a unique ${\mathbb A}^1$-fibration with a quasi-affine  base  
\[\mu\colon X\to Z\]
and $\SAut(X)$ consists of equivalent $\G_a$-actions that act by translations on fibers of $\mu$, see \cite[Section 6.1]{KPZ1}.
Moreover, $\Aut_{\alg}(X)=\TT\ltimes\U(X)$, where $\TT$ is an algebraic torus of dimension $\le\dim X$
and $\U(X)\subset \Aut(X)$ is an  abelian  ind-subgroup which is of infinite dimension if $\dim X \ge 2$. In particular, 
$\Aut_{\alg}(X)$ is a nested ind-group.
\item $\SAut(X)$ is trivial. Then $\Aut_{\alg}(X)$ is a torus, and there are no ${\mathbb A}^1$-fibrations with quasi-affine base.
\end{enumerate} 
\end{theorem}
\begin{proof}
First, assume that $\U(X)$ is non-trivial. Let us prove that the case (i) holds.

 Since all elements of $\Aut_{\text{alg}}(X)$ are algebraic, Proposition
\ref{nonlocfin} implies that all $\G_a$-actions on $X$ are equivalent.
It is well known that any non-trivial  $\G_a$-action $H=\{\exp(t\partial)\}$, where $\partial\in\LND(X)$,   induces an  $\mathbb{A}^1$-fibration  over an quasi-affine base $X/\!/H$.  Indeed, the invariant ring ${\mathcal O}(X)^{H}=\ker\partial$ is of codimension one in ${\mathcal O}(X)$, so $X\to\Spec{\mathcal O}(X)^{H}$ is a dominant morphism, whose general fibers are one-dimensional, irreducible and coincide with ${\mathbb A}^1$ by \cite[Cor. 1.29]{Fr}.  Conversely,  assume that there are two distinct $\mathbb{A}^1$-fibrations     $\pi_1\colon X \to B_1$ and $\pi_2\colon X \to B_2$  with quasi-affine bases $B_1$ and $B_2$. For  each fibration $\pi_i$ 
there exists an affine trivialization chart $U_i\subset B_i,$ $\pi_i^{-1}(U_i)\cong U_i\times {\mathbb A}^1$. Thus, in terms of  \cite{KPZ}, $X$ is cylindrical. Following \cite[Proposition 3.5]{KPZ} for both fibrations, we obtain two non-equivalent $\G_a$-actions.
This proves the first part of (i).

To prove the second part of  (i) we note that 
$\g=\t\oplus \u$ as a vector space by Proposition~\ref{th:lie}.
If $\dim X = 1$, then $X \simeq \mathbb{A}^1$ by \cite[Cor. 1.29]{Fr}. Otherwise,  by \cite[Principle 7]{Fr}, $\u$ contains an infinite-dimensional subspace $\{f\partial\mid f\in\ker\partial\}$ for any LND $\partial$.
  Moreover, $\u$ is graded by the character lattice of $\TT$, and  one can construct  an increasing sequence $\u_1\subset\u_2\subset\ldots\subset\u$ 
of finite-dimensional $\t$-stable subalgebras that  exhaust $\u$. 
So, we obtain a filtration by finite-dimensional Lie subalgebras
\[\g=\bigcup_{i=1}^\infty \t\oplus \u_i.\]

There exists a commutative unipotent subgroup $\U_i \subset \Aut(X)$ such that $\Lie \U_i = \u_i$ and $G_i=\TT\ltimes \U_i\subset\Aut X$ is an algebraic subgroup with 
the tangent Lie algebra $\t\oplus \u_i$.
We claim that 
$\Aut_{\text{alg}}(X) = \varinjlim G_i$. 
Indeed, for any connected algebraic subgroup $G \subset \Aut(X)$ we have $\Lie G \subset \t\oplus \u_i$, hence $G \subset G_i$ and the claim follows.

Now assume that $\U(X)$ is trivial, i.e., 
 $X$ does not admit a $\G_a$-action. By Proposition \ref{th:lie}, $\g=\t$, where $\t=\Lie\TT$ for a maximal subtorus $\TT\subset\Aut(X)$. Hence, $\Aut_{\alg}(X)=\TT$. 
 \end{proof}

\begin{proof}[Proof of  Theorem~\ref{main}]

Proposition \ref{nonlocfin}(2) and Proposition \ref{Prop4.4} provide the implications
$(2) \Rightarrow (1)$ and $(1) \Rightarrow (4)$
respectively.
The implications $(4)\Rightarrow (3)$ and $(3)\Rightarrow(2)$ are clear. The proof follows.
\end{proof}


\begin{thebibliography}{}
\bibitem{AG} I.~Arzhantsev, S.~Gaifullin, \emph{The automorphism group of a rigid affine variety},
Math.Nachrichten \textbf{290} (2017), no. 5--6, 662--671.


\bibitem{BD} J.~Blanc, A.~Dubouloz, {\em
Affine surfaces with a huge group of automorphisms}, Int.\ Math.\ Res.\ Notices   \textbf{2015} (2015), no. 2, 422--459.


\bibitem{Co} I. S. Cohen, \emph{On the structure and ideal theory of complete local rings}, Trans. Amer. Math. Soc. \textbf{59} (1946), no. 1, 54--106.

\bibitem{FZ-uniqueness} H.~Flenner, M.~Zaidenberg, \emph{On the uniqueness of $\CC^*$-actions on affine surfaces},  in: Affine Algebraic Geometry, Contemporary Mathematics, Vol. 369, Amer. Math. Soc. Providence, R.I., 2005.
\bibitem{FK} J.-P.~Furter,  H.~Kraft, \emph{On the geometry of the automorphism groups of affine varieties}, arXiv:1809.04175. 

\bibitem{Fr}
 G. Freudenburg, \emph{Algebraic Theory of Locally Nilpotent Derivations}, Encyclopaedia of Mathematical Sciences, Vol. 136, Invariant Theory and Algebraic Transformation Groups, VII, Springer-Verlag, Berlin, 2006.
 
 
\bibitem{GMM12}
 R.V.Gurjar, K.Masuda, M.Miyanishi, \emph{$\mathbb{A}^1$-fibrations on affine threefolds}, Journal of Pure and Applied Algebra \textbf{216} (2012), 296 -- 313.

\bibitem{KPZ}
T. Kishimoto, Y. Prokhorov, M. Zaidenberg, \emph{Group actions on affine cones},  Transformation Groups  \textbf{18} (2013),  1137--1153.

\bibitem{KPZ1} S.~Kovalenko, A.~Perepechko, M.~Zaidenberg,
\emph{On automorphism groups of affine surfaces}, in: Algebraic Varieties and Automorphism Groups, Advanced Studies in Pure Mathematics \textbf{75} (2017), 207--286.

\bibitem{Kr} H~Kraft, 
\emph{Automorphism Groups of Affine Varieties and a Characterization of Affine $n$-Space}, Trans. Moscow Math. Soc.  \textbf{78} (2017), 171--186.


\bibitem{Ku} S. Kumar, \textit{Kac-Moody Groups, Their Flag Varieties and Representation Theory},
   Progress in Mathematics, Vol. 204, Birkh\"auser Boston Inc., Boston, MA, 2002.

\bibitem{LZ} V.~Lin, M.~Zaidenberg, {\em Configuration Spaces of the Affine Line
and their Automorphism Groups}.
In: Automorphisms in Birational and Affine Geometry. Levico Terme, Italy, October 2012.
Springer Proceedings in Mathematics and Statistics, Vol. 79, Springer, 2014.

\bibitem{Mat58} T. Matsusaka, \emph{ Polarized varieties, fields of moduli and generalized Kummer varieties of polarized varieties},
 Amer. J. Math. \textbf{80} (1958), 45--82.

\bibitem{Mi} M.~Miyanishi, \emph{$G_a$-actions and completions},
Journal of Algebra \textbf{319} (2008), 2845--2854.

\bibitem{PZ}  A.~Perepechko, M.~Zaidenberg,
\emph{Automorphism groups of affine $\mathrm{ML}_2$-surfaces: dual graphs and Thompson groups}, 
preprint: in preparation.

\bibitem{Sh66}
I.~R. Shafarevich, \emph{On some infinite-dimensional groups}, Rend. Mat. e
  Appl. (5) \textbf{25} (1966), no.~1-2, 208--212.
  
  \bibitem{Win03} J. Winkelmann,
  \emph{Invariant rings and quasiaffine quotients}, Math. Z. \textbf{244},  (2003), 163--174.


\bibitem{Stacks} The {Stacks Project Authors}, \em{Stacks Project}, http://stacks.math.columbia.edu, 2016.
\end{thebibliography}
\end{document}